\documentclass[11pt]{amsart}

\usepackage{hyperref}
\usepackage{amsfonts,amssymb,amscd,amsmath,amsthm,enumerate,mathrsfs}
\usepackage{graphicx,type1cm,eso-pic}
\newcommand{\abs}{\vspace{12pt}}

\DeclareMathOperator{\vol}{vol}

\DeclareMathOperator{\id}{id}

\DeclareMathOperator{\const}{const.}

\DeclareMathOperator{\Iso}{Iso}
\DeclareMathOperator{\grad}{grad}

\DeclareMathOperator{\Hess}{Hess}

\DeclareMathOperator{\diag}{diag}

\DeclareMathOperator{\norm}{{{\|.\|}}}
\DeclareMathOperator{\la}{\langle}
\DeclareMathOperator{\ra}{\rangle}

\newcommand{\R}{\mathbb{R}}
\newcommand{\Z}{\mathbb{Z}}

\newcommand{\N}{\mathbb{N}}
\newcommand{\C}{\mathbb{C}}

\newcommand{\g}{\mathfrak{g}}

\renewcommand{\g}{\gamma}

\newcommand{\LL}{\mathcal{L}}
\renewcommand{\H}{\mathcal{H}}

\newcommand{\RR}{\mathcal{R}}
\newcommand{\J}{\mathcal{J}}

\newcommand{\M}{\mathcal{M}}
\newcommand{\G}{\mathcal{G}}

\newcommand{\A}{\mathcal{A}}

\newcommand{\e}{\varepsilon}
\newcommand{\del}{\xi}
\newcommand{\Om}{\Omega}
\newcommand{\al}{\alpha}
\newcommand{\om}{\omega}
\newcommand{\lam}{\lambda}
\newcommand{\sig}{\sigma}

\theoremstyle{plain}
\newtheorem{defn}{Definition}[section]
\newtheorem{lemma}[defn]{Lemma}
\newtheorem{prop}[defn]{Proposition}
\newtheorem{thm}[defn]{Theorem}
\newtheorem{cor}[defn]{Corollary}

\theoremstyle{definition}
\newtheorem{remark}[defn]{Remark}

\newcommand{\mane}{Ma\~n\'e}

\newcommand{\Poincare}{Poincar\'e }

\usepackage{cite}
\def\BibTeX{{\rm B\kern-.05em{\sc i\kern-.025em b}\kern-.08em
    T\kern-.1667em\lower.7ex\hbox{E}\kern-.125emX}}

\usepackage{tikz}
\usetikzlibrary{matrix,shapes,arrows}

\begin{document}

\hypersetup{pdftitle = {Minimal rays on surfaces of genus greater than one}, pdfauthor = {Jan Philipp Schr\"oder}}

\title[Minimal rays on surfaces]{Minimal rays on surfaces \\ of genus greater than one}
\author[J. P. Schr\"oder]{Jan Philipp Schr\"oder}
\address{Faculty of Mathematics \\ Ruhr University \\ 44780 Bochum \\ Germany}
\email{\url{jan.schroeder-a57@rub.de}}

\begin{abstract}
For Finsler metrics (no reversibility assumed) on closed orientable surfaces of genus greater than one, we study the dynamics of minimal rays and minimal geodesics in the universal cover. We prove in particular, that for almost all asymptotic directions the minimal rays with these directions laminate the universal cover and that the Busemann functions with these directions are unique up to adding constants. Moreover, using a kind of weak KAM theory, we show that for almost all types of minimal geodesics in the sense of Morse, there is precisely one minimal geodesic of this type.
\end{abstract}

\maketitle


\section{Introduction and main results}

We begin by fixing some notation. We assume throughout the paper that $M$ is a closed orientable surface of genus $>1$. On $M$, there exists a (hyperbolic) Riemannian metric $g_h$ of constant curvature $-1$. The Riemannian universal cover of $M$, denoted by $(X,g_h)$, is identified with the \Poincare disc model, i.e. $X=\{z\in\C : |z|<1\}$ and $(g_h)_z(v,w) = 4\cdot(1-|z|^2)^{-2}\la v,w\ra_{euc}$, where the geodesics $\g\subset X$ are circle segments meeting $S^1=\{z\in\C : |z|=1\}$ orthogonally. Let $d_h$ be the distance function on $X$ induced by $g_h$. We write $\Gamma\leq\Iso(X,g_h)$ for the group of deck transformations $\tau:X\to X$ with respect to the covering $X\to M$, which extend to naturally to the ``boundary at infinity'' $S^1$.

Let $\G$ be the set of all oriented, unparametrized geodesics $\g\subset X$ with respect to $g_h$. We can think of $\G$ as $S^1\times S^1-\diag$, associating to $\g\in\G$ its pair of endpoints $(\g(-\infty),\g(\infty))$ on $S^1$. Moreover, set
\[ \G_\pm(\del) := \{\g\in \G : \g(\pm\infty)=\del \}. \]

We consider any Finsler metric $F:TX\to\R$, which is assumed to be invariant under $\Gamma$ (cf. Definition \ref{def finsler}). We write $SX=\{F=1\}\subset TX$ and $c_v:\R\to X$ for the geodesic with respect to $F$ defined by $\dot c_v(0)=v$.

The main object of interest in this paper are rays and minimal geodesics, that is geodesics $c:[0,\infty)\to X, ~ c:\R\to X$, respectively that minimize the length with respect to $F$ between any of their points. In \cite{morse}, H. M. Morse studied the global behavior of minimal geodesics and showed that minimal geodesics lie in tubes around hyperbolic geodesics $\g\in\G$ of finite width $D$, where $D$ is a constant depending only on $F$ and $g_h$. In particular, minimal geodesics $c:\R\to X$ have well-defined endpoints $c(\pm\infty)$ at infinity, i.e. on $S^1$; the analogous result holds for rays. For $\del\in S^1, \g\in\G$ set
\begin{align*}
\RR_+(\del) & := \{v \in SX ~|~ c_v:[0,\infty)\to X \text{ is a ray, } c(\infty)=\del \}, \\
\M(\g) & := \{v \in SX ~|~ c_v:\R\to X \text{ is a minimal geodesic, } c(\pm\infty)=\g(\pm\infty) \}.
\end{align*}

Morse studied in particular the behavior of minimal geodesics in $\RR_+(\del)$, where $\del$ is fixed under a non-trivial group element $\tau\in\Gamma$; we will recall these results in Subsection \ref{section periodic}. However, Morse left open the finer structure in the asymptotic directions which are not fixed by an element of $\Gamma$ and the author is not aware of any other work in the literature in this direction. The purpose of this paper is to fill this gap, i.e. to study the structure of $\RR_+(\del)$ for general $\del\in S^1$. While we are still not able to give the structure for all $\del$, we will be able to do it for ``most'' $\del$.

Another novelty in this paper, compared to the work of Morse, is the use of Finsler instead of Riemannian metrics. It was observed by E. M. Zaustinsky \cite{zaustinsky}, that the results of Morse carry over to these much more general systems. Moreover, it is known that Finsler metrics can be used to describe the dynamics of arbitrary Tonelli Lagrangian systems in high energy levels, cf. \cite{contreras1}.

The sets $\M(\g)$ are bounded by two particular minimal geodesics. The following lemma is Theorem 8 of \cite{morse}.

\begin{lemma}[Bounding geodesics]\label{bounding geodesics morse}
for $\g\in\G$, there are two particular non-intersecting, minimal geodesics $c_\g^0,c_\g^1$ in $\M(\g)$, such that all minimal geodesics in $\M(\g)$ lie in the strip in $X$ bounded by $c_\g^0,c_\g^1$.
\end{lemma}

As a rule, we will always assume that $c_\g^1$ lies left of $c_\g^0$ with respect to the orientation of $\g$.

We can now state our first result, saying that for most asymptotic directions $\del\in S^1$, the bounding geodesics $c_\g^i$ of the various $\g\in\G_+(\del)$ do not intersect. The proof makes use of weak KAM theory, but does not rely on the group action by $\Gamma$. In fact, Theorem \ref{thm countable unstable} and its Corollary \ref{cor thm countable unstable} hold for any Finsler metric $F$ on $X$, which is uniformly equivalent to the norm of $g_h$ (cf. Definition \ref{def finsler}).

\begin{thm}\label{thm countable unstable}
For all but countably many $\del\in S^1$, the set
\[ L(\del) := \bigcup\{c_\g^i(\R) : \g\in \G_+(\del), i = 0,1\} \]
is a lamination of $X$ (i.e. no curves in $L(\del)$ intersect transversely).
\end{thm}

It would be desirable to know if the set $L(\del)$ is a lamination of $X$ for all $\del\in S^1$; we would then gain more insight in the structure of $\RR_+(\del)$ for all $\del$, cf. Theorem \ref{thm A} below.

From this we will deduce the following.

\begin{cor}[Uniqueness of minimal geodesics]\label{cor thm countable unstable}
For almost all $\g\in\G\cong S^1\times S^1-\diag$ with respect to the Lebesgue measure, the set $\M(\g)$ consists of precisely one minimal geodesic.
\end{cor}

For the statement of our next results, we need the following definitions.

\begin{defn}
A minimal ray $c:[0,\infty)\to X$ is called \emph{forward unstable}, if there is no minimal ray $c':[0,\infty)\to X$ with $c'(0)\in c(0,\infty)$ and $c'(\infty)=c(\infty)$, which is not a subray of $c$.
\end{defn}

\begin{remark}
\begin{itemize}
\item Instability of minimal geodesics was studied in a paper \cite{klingenberg} of W. Klingenberg, and some of our results appear in \cite{klingenberg}. However, \cite{klingenberg} contains errors and the results which we prove here, in particular the existence of unstable geodesics, remained open. Our work will be independent of \cite{klingenberg}.

\item It is easy to see (cf. Lemma \ref{partial instability bounding}) that the unique geodesic in $\M(\g)$ in Corollary \ref{cor thm countable unstable} is forward (and backward) unstable.
\end{itemize}
\end{remark}

\begin{defn}
For $\g\in\G, \del\in S^1$ define the \emph{(forward) width} of $\g, ~\del$, respectively by
\begin{align*}
w(\g)& := \liminf_{t\to\infty} d_h(c_\g^0(\R),c_\g^1(t)), \\
w(\del) &:= \sup \{ w(\g) : \g\in\G_+(\del) \}.
\end{align*}
\end{defn}

\begin{remark}\label{bem width def}
Due to the Morse Lemma (cf. Theorem \ref{morse lemma}), $w(\del)$ is bounded by the global constant $D<\infty$. As we will see, $w(\del)=0$ implies that all minimal rays in $\RR_+(\del)$ are forward unstable and in particular that the set $L(\del)$ in Theorem \ref{thm countable unstable} is a lamination of $X$, while $w(\g)=0$ implies the forward instability of all minimal geodesics in $\M(\g)$. Moreover, $w(\del)=0$ implies a uniqueness result for weak KAM solutions and Busemann functions. Hence, the main task in this paper is to show $w(\del)=0$ for directions $\del\in S^1$, which are not fixed by elements of $\Gamma$.
\end{remark}

While Theorem \ref{thm countable unstable} and its corollary did not depend on $\Gamma$, the following results rely strongly on the $\Gamma$-action. Assuming a certain behavior of a background geodesic $\g\in\G_+(\del)$ under $\Gamma$, we can calculate $w(\del)$.

\begin{thm}[Dense directions]\label{dense width 0}
If $\G_+(\del)$ contains a hyperbolic geodesic, which has a dense forward orbit in the hyperbolic unit tangent bundle of $M=X/\Gamma$, then $w(\del)=0$. In particular, for almost all $\del\in S^1$ with respect to the Lebesgue measure, we have $w(\del)=0$.
\end{thm}

Note that Theorem \ref{dense width 0} has the following analogue for the case where $M$ is the $2$-torus $\R^2/\Z^2$ (cf. \cite{bangert1} for the Riemannian, \cite{paper1} for the Finsler case): For all irrational rotation directions (which have full Lebesgue measure in $S^1$), the set of minimal geodesics with this direction is a lamination of the universal cover $\R^2$ and the weak KAM solutions with irrational directions are unique up to adding constants.

For completeness and to exemplify the concept of width, we state the following proposition, which follows directly from the results of Morse.

\begin{prop}[Periodic directions]\label{width periodic intro}
If $\g\in \G$ is an axis of a non-trivial group element $\tau\in\Gamma$ (i.e. $\tau \g=\g$), then
\begin{align*}
 w(\g) &  = w(\g(\infty)) = \inf_{t\in\R} d_h(c_\g^0(\R),c_\g^1(t)) , \\
 w(\g') & = 0 \qquad \forall \g'\in \G_+(\del)-\{\g\}.
\end{align*}
In particular, $w(\g(\infty))=0$ if and only if there is only one geodesic in $\M(\g)$. 
\end{prop}

A direct generalization of periodicity is recurrence; we say that $\g\in\G$ is forward recurrent, if the projection of $\dot \g$ has a forward recurrent orbit in the hyperbolic unit tangent bundle of $M=X/\Gamma$. For the case $w(\del)=0$ the structure of $\RR_+(\del)$ was explained in Remark \ref{bem width def} above. If $w(\del)>0$, we have the following results.

\begin{thm}[Recurrent directions]\label{thm intro recurrent}
Let $\del\in S^1$ and $w(\del)>0$.
\begin{enumerate}
 \item There are at most two forward recurrent geodesics in $\G_+(\del)$,
 
 \item If $\g\in \G_+(\del)$ is forward recurrent, then $w(\g)=w(\del)$ and there is a forward unstable geodesic in $\M(\g)$,

 \item If $\del$ satisfies the conclusion of Theorem \ref{thm countable unstable} (i.e. if the set of bounding geodesics $L(\del)$ with direction $\del$ is a lamination of $X$), then there can be at most one forward recurrent direction in $\G_+(\del)$ and if there is a forward recurrent direction $\g\in\G_+(\del)$, then
\begin{align*}
 w(\g) &  = w(\g(\infty)) = \inf_{t\in\R} d_h(c_\g^0(\R),c_\g^1(t)) , \\
 w(\g') & = 0 \qquad \forall \g'\in \G_+(\del)-\{\g\}.
\end{align*}
\end{enumerate}
\end{thm}

\begin{cor}
If $\g\in \G$ is forward recurrent, then there exists forward unstable minimal geodesic in $\M(\g)$.
\end{cor}

\begin{proof}
If $w(\g(\infty))=0$, then all minimal geodesics in $\M(\g)$ are forward unstable by Proposition \ref{w=0 implies unique weak KAM}. If $w(\g(\infty))>0$, then the claim follows from Theorem \ref{thm intro recurrent} (2).
\end{proof}

In Mather theory, one studies minimal geodesics $c:\R\to M$ (more generally: action minimizers of Lagrangian systems), such that $\dot c:\R\to SM$ is a graph over its projection in $M$. This means that in the universal cover, the curves $\tau c(\R)$ and $c(\R)$ are equal or disjoint for all $\tau\in\Gamma$, and $c$ is called \emph{simple}. Note that, if a minimal geodesic $c$ in $\M(\g)$ is simple, then so is its background geodesic $\g$. In this case we have the following result.

\begin{thm}[Simple directions]\label{thm simple}
If $\G_+(\del)$ contains a simple hyperbolic geodesic, which is not the axis of any $\tau\in\Gamma-\{\id\}$, then $w(\del)=0$.
\end{thm}

It is not clear to us whether the conclusion of Theorem \ref{thm countable unstable} is true for all $\del\in S^1$. If it would hold, then this would clarify the structure of all $\RR_+(\del)$, as seen in the following theorem.

\begin{thm}\label{thm A}
If $F$ is such that for all $\del\in S^1$ the set of bounding geodesics $L(\del)$ with direction $\del$ defined in Theorem \ref{thm countable unstable} is a lamination of $X$, then $w(\del)=0$ for all $\del\in S^1$ that are not fixed under any $\tau\in\Gamma-\{\id\}$.
\end{thm}

\begin{remark}
The set $L(\del)$ is a lamination of $X$ for all $\del\in S^1$, e.g., if $F$ has no conjugate points, cf. Theorem 12.1 in \cite{morse_hedlund}. In particular, if $F$ is Riemannian with non-positive curvature, then the only flat strips of uniformly positive width are periodic. By Theorem \ref{thm countable unstable}, this assumption of Theorem \ref{thm A} is always ``almost'' true.
\end{remark}

\begin{remark}\label{bem high dim}
The main concepts in this paper work in any dimension, in particular the weak KAM theory for manifolds of hyperbolic type developed in Section \ref{section weak KAM} and the concept of width in Section \ref{section width}. E.g., if in a closed manifold $M$ carrying a Riemanniam metric $g_h$ of strictly negative curvature there exists a hyperbolic geodesic $\g$ with respect to $g_h$, such that in $\M(\g)$ there is only one minimal geodesic, then $w(\del)=0$ for almost all $\del$ in the so-called Gromov boundary (with respect to the Lebesgue measure, identifying the Gromov boundary with a sphere $S^{\dim M}$) , cf. the arguments for the proof of Theorem \ref{dense width 0}. Here $w(\del)$ has to be defined differently, for instance by setting
\[ w(\del) = \sup \big\{ \liminf_{t\to\infty} d_h(c_v(\R), c_w(t)) : v,w\in \RR_+(\del) \big\} . \]
Most results, however, are strongest in dimension two, hence we stick to this case to simplicity the exposition.
\end{remark}

{\bf Structure of this paper} In Section \ref{section finsler}, we recall basic definitions and properties as well as the Morse Lemma. Section \ref{section weak KAM} is devoted to the proof of Theorem \ref{thm countable unstable}; here we study so-called weak KAM solutions, which are used to prove Theorem \ref{thm countable unstable} and its corollary. In Section \ref{section width}, the concept of width of asymptotic directions is introduced to obtain Theorems \ref{dense width 0}, \ref{width periodic intro}, \ref{thm intro recurrent}, \ref{thm simple} and \ref{thm A}. In Appendix \ref{appendix}, we make additional remarks on weak KAM solutions in dimension two; the results of the appendix are, however, not used in this paper.

\section{Finsler metrics and minimal rays}\label{section finsler}

We write $\pi:TX\to X$ for the canonical projection, $0_X$ denotes the zero section and $T_xX=\pi^{-1}(x)$ the fibers.

\begin{defn}\label{def finsler}
A function $F:TX\to\R$ is a \emph{Finsler metric on $X$}, if the following conditions are satisfied:
\begin{enumerate}
\item (smoothness) $F$ is $C^\infty$ off $0_X$, 

\item (positive homogeneity) $F(\lam v) = \lam F(v)$ for all $v\in TX, \lam\geq 0$,

\item (strict convexity) the fiberwise Hessian $\Hess(F^2|_{T_xX})$ of the square $F^2$ is positive definite off $0_X$ for all $x\in X$.
\end{enumerate}

We say that $F$ is \emph{uniformly equivalent} to $g_h$, if
\[ \exists c_F > 0: \quad \frac{1}{c_F}\cdot F \leq \norm_h \leq c_F\cdot F. \]

We say that $F$ is \emph{invariant under $\Gamma\leq \Iso(X,g_h)$}, the group of deck transformations $\Gamma$ with respect to the covering $X\to M$, if
\[ F(d\tau(\pi v)v)=F(v) \quad \forall v\in TX, \tau\in\Gamma. \]
\end{defn}

Note that uniform equivalence of $F, g_h$ is implied by invariance under $\Gamma$ due to compactness of $M$. For the rest of this paper, we fix the Finsler metric $F$ on $X$, which is assumed to be uniformly equivalent to $g_h$.

Write $SX=\{v\in SX : F(v)=1\}$ for the unit tangent bundle of $F$. The \emph{geodesic flow} $\phi_F^t:SX\to SX$ of $F$ is given by $\phi_F^tv=\dot c_v(t)$, where $c_v:\R\to X$ is the $F$-geodesic defined by $\dot c_v(0)=v$. We write $l_F(c) = \int_a^bF(\dot c) dt$ for the $F$-length of absolutely continuous ($C^{ac}$) curves $c:[a,b]\to X$ and
\[ d_F(x,y) = \inf\{l_F(c) ~|~ c:[0,1]\to X ~ C^{ac}, ~ c(0)=x,c(1)=y \} \]
for the $F$-distance. Note that if $F$ is not reversible, i.e. if not $F(\lam v)=|\lam|F(v)$ for all $\lam\in\R$, we have $d_F(x,y)\neq d_F(y,x)$ in general.

\begin{defn}
A $C^{ac}$ curve segment $c:[a,b]\to X$ with $F(\dot c)=1$ a.e. is said to be \emph{minimal}, if $l_F(c)=d_F(c(a),c(b))$. Curves $c:[0,\infty)\to X$, $c:(-\infty,0]\to X$, $c:\R\to X$ are called \emph{forward rays, backward rays, minimal geodesics}, respectively, if each restriction $c|_{[a,b]}$ is a minimal segment. Set
\begin{align*}
\RR_- & := \{v \in SX : c_v:(-\infty,0]\to X \text{ is a backward ray} \}, \\
\RR_+ & := \{v \in SX : c_v:[0,\infty)\to X \text{ is a forward ray} \}, \\
\M & := \{v \in SX : c_v:\R\to X \text{ is a minimal geodesic} \}.
\end{align*}
\end{defn}

We will in this paper mainly be concerned with forward rays, the results for backward rays being completely analogous.

The following lemma is a key property of rays. It excludes in particular successive intersections of rays and shows that asymptotic rays can cross only in a common initial point. The idea of the proof is classical; it can be found in \cite{paper1}, Lemma 2.21.

\begin{lemma}\label{crossing minimals}
Let $v_n,v,w\in \RR_+$ with $v_n\to v$. Assume that $\pi w=c_v(a)$ for some $a>0$, but $w\neq \dot c_v(a)$. Then, for all $\delta>0$ and sufficiently large $n$,
\[ \inf \{ d_h(c_{v_n}(s),c_w(t)) : s\in [a,\infty), t\in [\delta,\infty) \} > 0. \]
\end{lemma}

\subsection{The Morse Lemma and asymptotic directions}\label{section morse}

The following theorem, due to H. M. Morse, is the cornerstone of our investigation and arises from the uniform equivalence of $F$ and the negatively curved $g_h$. The fact that the Morse Lemma also holds in the Finsler case was first observed by E. M. Zaustinsky \cite{zaustinsky}. Due to Klingenberg \cite{klingenberg}, the Morse Lemma holds in any dimension.

\begin{thm}[Morse Lemma, cf. \cite{morse}]\label{morse lemma}
Let $F, g_h$ be uniformly equivalent. Then there exists a constant $D\geq 0$ depending only on $F, g_h$ with the following property. For any two points $x,y\in X$, the hyperbolic geodesic segment $\g:[0,d_h(x,y)]\to X$ from $x$ to $y$ and any minimal segment $c:[0,d_F(x,y)]\to X$ from $x$ to $y$ we have
\[ \max_{t\in[0,d_F(x,y)]} d_h(\g[0,d_h(x,y)],c(t)) \leq D. \]
If $c:[0,\infty)\to X$ is a forward ray, then there exists a hyperbolic ray $\g:[0,\infty)\to X$, and conversely, if $\g:[0,\infty)\to X$ is a hyperbolic ray, then there exists a forward ray $c:[0,\infty)\to X$, such that
\[ \sup_{t\in[0,\infty)} d_h(\g[0,\infty),c(t)) \leq D. \]
The analogous statements hold for backward rays and minimal geodesics.
\end{thm}

With the Morse Lemma we can associate asymptotic directions to rays. Namely, if $c:[0,\infty)\to X$ is a ray, choose any hyperbolic ray $\g:[0,\infty)\to X$, such that $\sup_{t\in[0,\infty)} d_h(\g[0,\infty),c(t)) <\infty$ and write $c(\infty):= \g(\infty)\in S^1$. Note that $c(\infty)=\del\in S^1$ if and only if $c(t)\to \del$ in the euclidean sense in $\C\supset X$. Recall that $\G$ is the set of oriented, unparametrized hyperbolic geodesics $\g\subset X$ and that $\RR_\pm, \M$ are the sets of initial vectors of forward and backward rays and minimal geodesics, respectively.

\begin{defn}
For $\del\in S^1, \g\in \G$ we set
\begin{align*}
 \RR_\pm(\del) & := \{ v\in \RR_\pm : c_v(\pm\infty)=\del \}, \\
 \M(\g) & := \{ v\in \M : c_v(-\infty)=\g(-\infty) \text{ and } c_v(\infty)=\g(\infty) \}.
\end{align*}
\end{defn}

\begin{remark}\label{remark bounding geod}
Due to the Morse Lemma we have
\[ \RR_\pm = \bigcup_{\del\in S^1} \RR_\pm(\del), \quad \M = \bigcup_{\g\in\G} \M(\g).\]
\end{remark}

The following lemma shows that the usual topology of $S^1$ is connected with the topology in $SX$. Recall that we identified $X$ with the interior of the unit disc in $\C$. For the closure $\bar X = X\cup S^1$ we take the usual topology induced by $\C$.

\begin{lemma}\label{as dir cont rays}
If $v_n,v\in\RR_+$ with $v_n\to v$, then $c_{v_n}(\infty)\to c_v(\infty)$ in $S^1$.
\end{lemma}

\begin{proof}
Let $\g_n,\g\in\G$ be the hyperbolic geodesics connecting $\pi v_n, \pi v$ to $c_{v_n}(\infty), c_v(\infty)$, respectively. The Morse Lemma shows that for large $n$ the hyperbolic geodesics $\g_n, \g$ stay at bounded $d_h$-distance on long subsegments, their length increasing to $\infty$ with $n$. In the limit, by $\pi v_n\to \pi v$, we obtain a hyperbolic limit geodesic initiating from $\pi v$, with bounded distance from $c_v[0,\infty)$ and hence from $\g$. By the uniqueness of such $\g$, we obtain $\g_n\to \g$ with respect to the euclidean Hausdorff distance in $\C$ and in particular their endpoints converge in $S^1$.
\end{proof}

\section{Weak KAM solutions}\label{section weak KAM}

We continue to assume that $F,g_h$ are uniformly equivalent. For the proof of Theorem \ref{thm countable unstable}, we will use certain functions, that have been studied in various situations throughout the literature; in particular, they arise as Busemann functions, as we shall see below. We use the language of A. Fathi's weak KAM theory (cf. \cite{fathi} for an introduction).

\begin{defn}
 A function $u:X\to\R$ is called a \emph{forward weak KAM solution of direction $\del\in S^1$}, written $u\in\H_+(\del)$, if the following two conditions hold:
\begin{enumerate}
 \item $u(y)-u(x)\leq d_F(x,y)$ for all $x,y\in X$,

 \item for all $x\in X$ there exists a forward ray $c:[0,\infty)\to X$ with $c(0)=x$, $c(\infty)=\del$, $F(\dot c)=1$ and $u(c(t))-u(x) = t$ for all $t\in[0,\infty)$.
\end{enumerate}
We write $\J_+(u)$ for the set of all $v\in SX$, such that $u(c_v(t))-u(\pi v) = t$ for all $t\geq 0$, call the ray $c_v:[0,\infty)\to X$ \emph{$u$-calibrated} and set
\[ \J(u) := \bigcap_{t\geq 0} \phi_F^t\J_+(u). \]
The set of all forward weak KAM solutions is denoted by
\[ \H_+ := \bigcup_{\del\in S^1} \H_+(\del). \]
We write
\[ \om: \H_+ \to S^1, \quad  \om(u) = \del ~ :\iff ~ \J_+(u)\subset \RR_+(\del) \]
and call $\om(u)$ the \emph{asymptotic direction} of $u\in \H_+$.

The set of backward weak KAM solutions $\H_-$ is defined analogously, with analogous sets $\J_-(u)$ and $\J(u) = \cap_{t\leq 0}\phi_F^t\J_-(u)$ for $u\in\H_-$ and the asymptoptic direction
\[ \al: \H_- \to S^1, \quad  \al(u) = \del ~ :\iff ~ \J_-(u)\subset \RR_-(\del) . \] 
\end{defn}

We will show in Lemma \ref{as dir well-def}, that $\om:\H_+\to S^1$ is well-defined. Note that it follows from the definition (condition (1)) that $\J_+(u)\subset \RR_+$ and $\J(u)\subset \M$ for all $u\in\H_+$.

The following lemma is well-known in weak KAM theory and probably also in the theory of Busemann functions. If $u:X\to\R$ is differentiable in $x\in X$, write
\[ \grad_F u(x) := \LL_F^{-1}(du(x)), \quad \text{where } \LL_F(v) = \frac{1}{2} d_vF^2(v) \in T_{\pi v}^*X \subset T^*X. \]
Here $d_v$ denotes differentiation along the fiber, $\LL_F:TX\to T^*X$ is the Legendre transform associated to $F$. Note that if $F$ is Riemannian, then $\grad_Fu$ is the usual gradient.

\begin{lemma}\label{lemma fathi}
Let $u\in \H_+$. Then for all $t>0$, $u$ is differentiable in $\pi\phi_F^t\J_+(u)$. If $u$ is differentiable in $x\in X$, then
\[ \J_+(u)\cap T_xX = \{ \grad_Fu(x) \}. \]
\end{lemma}

Note that $\J_+(u)$ is forward-invariant by the geodesic flow $\phi^t_F$ and that by Lemma \ref{lemma fathi} the set $\phi_F^\e\J_+(u)$ for $\e>0$ is a graph over the zero section $0_X$ (it is even locally Lipschitz by Theorem 4.11.5 in \cite{fathi}). We will refer to this fact as the \emph{graph property} of $\J_+(u)$.

\begin{proof}
Define the Lagrangian $L=\frac{1}{2}(F^2+1)$, then $L$ is a Tonelli Lagrangian (the fact that it is only $C^1$ in $0_X$ does not matter in this situation). Write $A_L(c)=\int L(\dot c)dt$ for the action and consider \mane's potential
\[ \Phi_L(x,y) := \inf\left\{ A_L(c) : T>0, c:[0,T]\to X ~ C^{ac}, c(0)=x,c(T)=y \right\}. \]

We claim that $\Phi_L=d_F$. This shows that $u$ is dominated by $L$ in the sense of Fathi \cite{fathi}, such that the lemma is just a reformulation of Theorem 4.3.8 in \cite{fathi}.

Proof of the claim: Observe that the action of any closed curve in $X$ with respect to $L-1/2$ is non-negative and, using Proposition 5.11 in \cite{sorrentino}, one finds for each $x,y\in X$ a time $T>0$ and a $C^{ac}$ curve $c:[0,T]\to X$ with $c(0)=x,c(T)=y$, such that $A_L(c)=\Phi_L(x,y)$. The energy function of $F^2/2$ is just $F^2/2$ itself and $c$ is also critical for the action with respect to $F^2/2$ when fixing the connection time, hence we obtain $F^2(\dot c)=\const$. Consider for $s>0$ the reparametrizations $c_s:[0,T/s]\to X$, $c_s(t) = c(st)$. By minimality of $c=c_1$ under all $c_s$ and homogeneity of $F$, one easily shows
\[ A_L(c_s)=T/2 \cdot(s F^2(\dot c) + 1/s), \quad 0 = \frac{d}{ds}\bigg|_{s=1} \frac{2}{T} A_L(c_s) = F^2(\dot c) - 1, \]
and hence $F(\dot c)=1$. Thus, by minimality of $c$ with respect to $A_L$
\begin{align*}
d_F(x,y) & \leq l_F(c) = \int_0^T 1 dt = \int_0^T L(\dot c) dt = \Phi_L(x,y).
\end{align*}
On the other hand, if $c:[0,d_F(x,y)]\to X$ is a minimal geodesic segment from $x$ to $y$ with $F(\dot c)=1$, we obtain the other inquality:
\[ d_F(x,y) = l_F(c) = \int_0^{d_F(x,y)} 1 dt = \int_0^{d_F(x,y)} L(\dot c) dt \geq \Phi_L(x,y) . \]
\end{proof}

\begin{lemma}\label{as dir well-def}
If $u\in\H_+(\del)$, then $\J_+(u)\subset \RR_+(\del)$.
\end{lemma}

\begin{proof}
Let $v\in \J_+(u)$ and $t>0$. By definition, there exists a minimal ray $c:[0,\infty)\to X$ with $\dot c\in\J_+(u)$ and $c(0)=c_v(t), c(\infty)=\del$. By Lemma \ref{lemma fathi} we find $\dot c(0)=\dot c_v(t)$, i.e. also $c_v(\infty)=c(\infty)=\del$.
\end{proof}

We have the following corollary to Lemma \ref{lemma fathi}.

\begin{cor}\label{J=J' then u=u'}
If $u,u'\in\H_+$ with $\J_+(u)=\J_+(u')$, then $u-u'$ is constant.
\end{cor}

\begin{proof}
Let $U\subset X$ be the set where both $u,u'$ are differentiable. $U$ has full measure by Rademacher's Theorem (weak KAM solutions are Lipschitz by definition). Lemma \ref{lemma fathi} and $\J_+(u)=\J_+(u')$ show $du(x)=du'(x)$ for all $x\in U$. The claim follows.
\end{proof}

We now show $\H_+\neq \emptyset$. There is a classical way to construct weak KAM solutions by considering so-called Busemann functions or horofunctions. For this, let $x_0\in X$ and $x_n\in X$ be a sequence with $d_h(x_0,x_n)\to \infty$. Any $C_{loc}^0$ limit $u\in C^0(X)$ of the sequence of functions
\[ x\mapsto d_F(x_0,x_n) - d_F(x,x_n)  \]
is called a horofunction. If $v\in \RR_+$ and $x_n=c_v(n)$, then $u$ is called the Busemann function of $c_v$.

\begin{lemma}\label{existence horofunction}
If $x_0,x_n\in X$ with $d_h(x_0,x_n)\to\infty$, then for the sequence of functions $d_F(x_0,x_n) - d_F(. ,x_n) : X\to \R$ there exist $C_{loc}^0$ limit functions $u\in C^0(X)$. Moreover, any limit $u$ belongs to $\H_+$.
\end{lemma}

In particular, associating the Busemann function $u_v\in\H$ to $v\in\RR^+$ shows
\[ \RR_+(\del) = \bigcup_{u\in\H_+(\del)} \J_+(u) . \]

\begin{proof}
There exists a constant $c_F$, such that $\frac{1}{c_F} d_F \leq d_h \leq c_F\cdot d_F$ due to uniform equivalence of $F,g_h$ by assumption. Let $u_n := d_F(x_0,x_n) - d_F(. ,x_n)$, then by the triangle inequality
\[ u_n(x)-u_n(y) = d_F(y ,x_n)- d_F(x ,x_n) \leq d_F(y,x)\leq c_F\cdot d_h(x,y) \]
and by symmetry of $d_h$, the functions $u_n$ are Lipschitz with Lipschitz constant $c_F$. Using $u_n(x_0)=0$ for all $n$, the Arzela-Ascoli Theorem shows that the $u_n$ have a $C_{loc}^0$ convergent subsequence with limit in $C^0(X)$. After passing to a subsequence, we now assume $u=\lim u_n$. We then obtain for $x,y\in X$ as before
\[ u(y) - u(x) = \lim u_n(y)-u_n(x) \leq d_F(x,y). \]
Now let $x\in X$ and $c_n:[0,d_F(x,x_n)]\to X$ be a minimal segment from $x$ to $x_n$ with $F(\dot c_n)=1$. Choose a convergent subsequence of $\dot c_n(0)$ with limit $v$, then $c_v:[0,\infty)\to X$ is a minimal ray. Similarly, let $c_n^0$ connect $x_0$ to $x_n$ with $\dot c_n(0)\to v_0$. All $c_n,c_n^0$ and hence the limits $c_v,c_{v_0}$ have uniformly bounded distance by the Morse Lemma. Setting $\del:=c_{v_0}(\infty)$, we have $c_v(\infty)=c_{v_0}(\infty)=\del$ for all so obtained $v$. We have to show $u(c_v(t)) - u(x) = t$ for $t\geq 0$. The triangle inequality for $d_F$ and $c_n(t)\to c_v(t)$ for $n\to\infty$ show $d_F(c_n(t),x_n)-d_F(c_v(t),x_n)\to 0$. Hence, using minimality of $c_n$, we have
\begin{align*}
&~ u(c_v(t)) - u(x) = \lim u_n (c_v(t)) - u_n(x) = \lim d_F(x ,x_n)- d_F(c_v(t) ,x_n) \\
= &~ \lim d_F(x ,x_n)- d_F(c_v(t) ,x_n) + d_F(c_n(t),x_n) - d_F(c_n(t),x_n) \\
= &~ \lim d_F(x ,x_n) - (d_F(x,x_n) - t) = t.
\end{align*}
\end{proof}

So far, everything in this section works for manifolds of hyperbolic type (i.e. possessing a metric $g_h$ of negative curvature) in any dimension. For the next proposition, we need $\dim M=2$. Recall the definition of the bounding geodesics $c_\g^0,c_\g^1$ of $\M(\g)$ in Lemma \ref{bounding geodesics morse}. In each $\J(u)$ there are similar bounding geodesics.

\begin{prop}\label{bounding geodesics for rays}
Given $u\in\H_+$ and $\g\in\G_+(\om(u))$, there exist two particular non-intersecting, minimal geodesics $c_{\g,u}^0,c_{\g,u}^1$ in $\J(u)\cap\M(\g)$, such that all minimal geodesics in $\J(u)\cap\M(\g)$ lie between $c_{\g,u}^0,c_{\g,u}^1$. Moreover, for any $x\in X$, there exists a unique $\g\in\G_+(\om(u))$, such that $x$ lies between $c_{\g,u}^0,c_{\g,u}^1$.
\end{prop}

\begin{proof}
We first show that for each $\g\in\G_+(\om(u))$, there is at least one geodesic in $\J(u)\cap\M(\g)$. For this, we take $x_n=\g(-n)$ and a ray $c_n:[0,\infty) \to X$ in $\J_+(u)$ initiating from $x_n$. Each $c_n$ has bounded distance from $\g$, independently of $n$, by the Morse Lemma. With $n\to\infty$, we obtain as a limit a minimal geodesic in $\J(u)\cap\M(\g)$. By the closedness of $\pi(\J(u)\cap\M(\g))$ and the graph property of $\J(u)$, there are rightmost and leftmost geodesics $c_{\g,u}^0,c_{\g,u}^1$ in $\J(u)\cap\M(\g)$; the first claim follows.

Now let $x\in X$ and assume that $x\notin \pi\J(u)$. Then there are two minimal geodesics $c_0,c_1$ in $\J(u)$ lying closest to $x$ on either side by the closedness of $\pi\J(u)$. We have to show that $c_0(-\infty)=c_1(-\infty)$ and assume the contrary. Choose then a hyperbolic geodesic $\g'\in \G_+(\om(u))$ with $\g'(-\infty)$ in the open segment $\sig\subset S^1$ between $c_i(-\infty)$, such that $\om(u)\notin\sig$. There exists $v'\in \J(u)\cap \M(\g')$, such that $c_{v'},c_0,c_1$ cannot intersect and $c_{v'}$ lies between $c_0,c_1$, contradicting the assumption that $c_0,c_1$ are closest to $x$.
\end{proof}

\subsection{Unstable geodesics}\label{section unstable}

We repeat the definition of instability.

\begin{defn}
A forward ray $c:[0,\infty)\to X$ is called \emph{forward unstable}, if for all minimal rays $c':[0,\infty)\to X$ with $c'(0)\in c(0,\infty)$ and $c'(\infty)=c(\infty)$ we have $c'[0,\infty)\subset c(0,\infty)$, i.e. $c'$ is a subray of $c$.

Backward instability of backward rays $c:(-\infty,0]\to X$ is defined analogously. A minimal geodesic is called \emph{unstable}, if it is both forward and backward unstable.
\end{defn}

Instability can be characterized by weak KAM solutions.

\begin{prop}\label{unstable Aubry}
For $v\in\RR_+(\del)$, the minimal ray $c_v:[0,\infty)\to X$ is forward unstable if and only if
\[ v\in \A_+(\del) :=  \bigcap_{u\in\H_+(\del)} \J_+(u). \]
In particular, the set of forward unstable rays $\A_+(\del)\subset\RR_+(\del)$ is closed.
\end{prop}

\begin{remark}
A similar set $\A_+(\del)$ appears in the setting of A. Fa\-thi's weak KAM theory. We called it the \emph{forward Aubry set of direction $\del$}. It is not clear to us, whether $\A_+(\del)\neq\emptyset$ for all $\del$. We do not expect any kind of continuity of $\del\mapsto \A_+(\del)$.
\end{remark}

\begin{proof}
Let $v\in\RR_+(\del)$ be forward unstable, $t>0$ and $u\in\H_+(\del)$. For $x=c_v(t)$, we find a ray $c:[0,\infty)\to X$ in $\J_+(u)$, such that $c(0)=x$. But due to $\J_+(u)\subset \RR_+(\del)$ and the instability of $c_v$, we have $c[0,\infty)=c_v[t,\infty)$, showing by closedness of $\J_+(u)$, that $v\in\J_+(u)$.

Conversely, let $v\in\A_+(\del)$, $t>0$ and assume that $c:[0,\infty)\to X$ is a ray in $\RR_+(\del)$ initiating from $c(0)=c_v(t)$. The Busemann function $u$ of $c$ then belongs to $\H_+(\del)$, while by assumption $v\in \J_+(u)$. Lemma \ref{lemma fathi} shows differentiability of $u$ in $c_v(t)$ and $\dot c(0)=\dot c_v(t)$, i.e. $c$ is a subray of $c_v$.
\end{proof}

For later reference, we prove the following lemma, showing a partial instability of the bounding geodesics $c_\g^i$ of $\M(\g)$ from Lemma \ref{bounding geodesics morse}. The lemma was basically already known to Morse.

\begin{lemma}\label{partial instability bounding}
If $\g\in \G$ and $S\subset X$ is the closed strip between $c_\g^0,c_\g^1$, then any ray $c:[0,\infty)\to X$ initiating in $S$ with $c(\infty)=\g(\infty)$ lies entirely in $S$.

The analogous statement holds for backward rays with $c(-\infty)=\g(-\infty)$.
\end{lemma}

\begin{proof}
Suppose $c:[0,\infty)\to X$ is a ray with $c(0)\in c_\g^1$ and $c(\infty)=\g(\infty)$, leaving $S$. Consider a sequence of minimal segments $c_n$ from $c_\g^1(-n)$ to $c(n)$, converging to a minimal geodesic $c$ in $\M(\g)$ by the Morse Lemma. But due to minimality of $c_\g^1, c$, it has to lie left of $c_\g^1$, contradicting the fact that all minimal geodesics in $\M(\g)$ lie right of $c_\g^1$ by definition.
\end{proof}

\subsection{Proof of Theorem \ref{thm countable unstable} and its corollary}

We prove a lemma. Recall the definition of $c_{\g,u}^i$ in Proposition \ref{bounding geodesics for rays} for $\g\in\G, u\in \H_\pm, i\in\{0,1\}$.

\begin{lemma}\label{countable 1}
Let $\del\in S^1$ and $u\in \H_-(\del)$, then for all but countably many $\g\in\G_-(\del)$ we have $c_{\g,u}^0=c_{\g,u}^1$, which is then forward unstable.
\end{lemma}

\begin{proof}
Let
\begin{align*}
 L &:= \bigcup \{ c_{\g,u}^i(\R) : \g\in \G_-(\del), i=0,1\} , \\
 A &:= \{ C\subset X : C \text{ connected component of $X-L$}\}. 
\end{align*}
$L$ defines a closed lamination of $X$ by the graph property of $\J_-(u)$, hence all open sets $C\in A$ are pairwise disjoint. Moreover, for any $\g\in \G_-(\del)$ with $c_{\g,u}^0 \neq c_{\g,u}^1$, the strip between the $c_{\g,u}^i$ is an element of $A$. The first part of the lemma follows, as long as $A$ is countable. But as $X$ is the union of countably many compact sets $K_n$, and each $K_n$ can contain at most countably many disjoint open sets. The claim follows.

For the forward instability of $c:=c_{\g,u}^0=c_{\g,u}^1$, suppose $c'$ is a forward ray in $\RR_+(\g(\infty))$ initiating from $c$. The nearby $c_{\g',u}^i$ converge to $c$, as $\g'\in\G_-(\del)$ converges to $\g$, while they have different points at $+\infty$. Hence $c'$ would have to be intersected twice by a minimal geodesic, contradiction.
\end{proof}

\begin{proof}[Proof of Theorem \ref{thm countable unstable}]
Take a dense sequence $\del_n\in S^1$ and write $A_n\subset S^1$ for the set of points $\g(\infty)$ of $\g\in \G_-(\del_n)$ and such that $\M(\g)$ contains a forwards unstable minimal geodesic. Then $A_n$ has a countable complement in $S^1$ by Lemma \ref{countable 1}. Set
\[A := \bigcap_{n\in\N} A_n, \]
which still has a countable complement in $S^1$. If $\del\in A$, then for all $n$, there exists a forward unstable, minimal geodesic $\M(\g_n)$, where $\g_n$ is the hyperbolic geodesic from $\del_n$ to $\del$. If $\g\in\G_+(\del)$ is arbitrary, choose two subsequences $\del_m^\pm$ of $\del_n$, such that $\del_m^\pm\to \g(-\infty)$ with $\del_m^-<\g(-\infty)<\del_m^+$ in the counterclockwise orientation of $S^1$. Writing $c_m^\pm:\R\to X$ for the forward unstable geodesic found in $\M(\g_m^\pm)$, we find limits $c^\pm$ of $\{c_m^\pm\}$ in $\M(\g)$. Since the $c_m^\pm$ cannot intersect the bounding geodesics $c_\g^i$ by instability, we have $c^-=c_\g^1, c^+=c_\g^0$. Since $\A_+(\del)$ is closed, the theorem follows.
\end{proof}

\begin{proof}[Proof of Corollary \ref{cor thm countable unstable}]
Let $A^+\subset S^1$ be the set of $\del$ satisfying the conclusion of Theorem \ref{thm countable unstable}, then $A^+$ has full Lebesgue measure, since it has a countable complement. For $\del\in A^+$ and any $u\in \H_+(\del)$, both $c_\g^i=c_{\g,u}^i$ belong to $\J(u)$ by construction, so we can apply Lemma \ref{countable 1}, showing $c_\g^0=c_\g^1$ for all but countably many $\g\in \G_+(\del)$. Hence, if we write $A_\del^-$ for the set of points $\g(-\infty)$ of $\g\in\G_+(\del)$ with $\M(\g)$ consisting of only one minimal geodesic, each $A_\del^-$ with $\del\in A^+$ again has full Lebesgue measure in $S^1$. We obtain
\begin{align*}
&~ \vol(\{(\del_-,\del_+) \in S^1\times S^1 : \del_+\in A^+, \del_-\in A_{\del_+}^- \} ) \\
= &~ \int_{A^+} \left(\int_{A^-_{\del_+}} 1 d\del_- \right) d\del_+ = \int_{S^1} \left(\int_{S^1} 1 d\del_-  \right) d\del_+ \\
= &~ \vol(S^1\times S^1).
\end{align*}
\end{proof}

\section{The width of asympotic directions}\label{section width}

For the whole Section \ref{section width}, we assume that $F$ is invariant under the group $\Gamma$ of deck transformations with respect to the covering $X\to M$. Note that then $F$ can be considered a Finsler metric on $M$. We concentrate of the forward behavior of forward rays, while all results have analogons for backward rays.

As we remarked earlier, the Morse Lemma is the cornerstone of this work. It allows us to define the width $w(\del)$ of $\del\in S^1$, which is special for surfaces of genus $>1$. It does not work for the 2-torus, even though the Morse Lemma also holds in this case due to G. A. Hedlund \cite{hedlund}, but does work in closed manifolds of arbitrary dimension admitting a Riemannian metric $g_h$ of negative curvature. Hence the finiteness of the following objects reflects once more the hyperbolic background structure of $X$.

Recall the definition of the bounding geodesics $c_\g^i$ of $\M(\g)$ in Lemma \ref{bounding geodesics morse}.

\begin{defn}\label{def width}
For $\del\in S^1, \g\in\G$ set
\begin{align*}
i(\g) &:= \inf_{t\in\R} d_h(c_\g^0(\R),c_\g^1(t)), \\
w(\g) &:= \liminf_{t\to\infty} d_h(c_\g^0(\R),c_\g^1(t)), \\
w(\del) &:= \sup_{\g\in\G_+(\del)} w(\g), \\
\end{align*}
We call $w(\g),w(\del)$, the \emph{(forward) width} of $\g,\del$, respectively.
\end{defn}

\begin{remark}\label{width bounded}
Let
\[ D:= \sup\{ d_h(c_\g^0(\R),c_\g^1(t)) : \g\in\G, t\in \R \}, \]
which is finite by the Morse Lemma. Then
\[ i(\g)\leq w(\g)\leq w(\g(\infty)) \leq D \quad \forall \g\in\G. \]
\end{remark}

One can study uniqueness of weak KAM solutions using the notion of width as follows.

\begin{prop}\label{w=0 implies unique weak KAM}
If $w(\del)=0$ for $\del\in S^1$, then there is up to adding a constant only one forward weak KAM solution $u\in\H_+(\del)$, i.e. for any $u\in\H_+(\del)$ we have
\[ \H_+(\del) = \{u + c : c\in\R \}. \]
In particular, if $w(\del)=0$, then $\A_+(\del)=\RR_+(\del)$, i.e. all rays $c:[0,\infty)\to X$ with $c(\infty)=\del$ are forward unstable.
\end{prop}

\begin{proof}
First observe, that all $c_\g^i$ with $\g\in\G_+(\del)$ are positively unstable. Namely, if $c:[0,\infty)\to X$ is a ray initiating e.g. from $c_\g^0$, then $c(0,\infty)$ has to lie inside the strip between $c_\g^0,c_\g^1$ by Lemma \ref{partial instability bounding}; now $w(\g)=0$ and Lemma \ref{crossing minimals} show that $c$ is a subray of $c_\g^0$. This shows that the bounding geodesic $c_{\g,u}^i$ in $\J(u)$ for $u\in \H_+(\del)$ (cf. Proposition \ref{bounding geodesics for rays}) coincide with $c_\g^i$.

Let now $u,u'\in\H_+(\del)$, $v\in \J_+(u)$, $t>0$ and $v'\in \J_+(u')$ with $\pi v'=c_v(t)$. By the instability of the $c_\g^i$, the rays $c_v$ and $c_{v'}$ lie in a common gap of $\A_+(\del)$, and hence between a pair $c_\g^0,c_\g^1$ for some $\g\in\G_+(\del)$. Again $w(\g)=0$ and Lemma \ref{crossing minimals} show that $c_{v'}$ is a subray of $c_v$. By closedness of $\J_+(u')$ we find $v\in \J_+(u')$, showing $\J_+(u)=\J_+(u')$ by symmetry. Corollary \ref{J=J' then u=u'} shows that $u-u'$ is constant.
\end{proof}

The widths defined above satisfy a useful upper semi-continuity property, from which we can deduce Theorem \ref{dense width 0} from the introduction. For this, we use the group action by $\Gamma$.

\begin{defn}\label{def positive gamma}
We say that a sequence $\g_n\in\G$ converges to $\g\in\G$, if for the pairs of endpoints $(\g_n(-\infty),\g_n(\infty))\to (\g(-\infty),\g(\infty))$ with respect to the euclidean metric in $\C\supset S^1$. 

A sequence $\tau_n\in\Gamma$ is \emph{positive for $\g\in\G$}, if there exists a sequence $x_n\in \g$ with $x_n\to \g(\infty)\in S^1$ (with respect to the euclidean distance in $\C\supset X$) and a compact set $K\subset X$, such that $\tau_nx_n\in K$ for all $n$.
\end{defn}

Intuitively, a sequence $\tau_n$ is positive for $\g\in \G$, if $\tau_n\g$ describes the behavior of $\g(t)$ in the compact quotient $M$, as $t\to\infty$.

\begin{lemma}\label{width semi-cont}
Let $\g,\g'\in\G$ and $\tau_n\in \Gamma$ be a positive sequence for $\g$, such that $\tau_n \g \to \g'$. Then $w(\g(\infty)) \leq i(\g')$.
\end{lemma}

\begin{proof}
Let $\g(\infty)=\del$. As $\tau_n$ is positive for $\g$, any other $\g''\in\G_+(\del)$ will converge under the $\tau_n$ to the same $\g'$ ($\g,\g''$ are asymptotic with respect to $d_h$, which is invariant under the isometries $\tau_n$). Hence it is enough to show $w(\g)\leq i(\g')$. Under $\tau_n$, the minimal geodesics $c_\g^0,c_\g^1$ have as a limit two geodesics $c^0,c^1$ in $\M(\g')$ (Theorem 7 in \cite{morse}), which have
\[ \inf_{t\in \R} d(c^0(\R),c^1(t)) \leq \inf_{t\in \R} d(c_{\g'}^0(\R),c_{\g'}^1(t)) = i(\g'). \]
Since $\tau_n$ is positive, the claim follows.
\end{proof}

\begin{proof}[Proof of Theorem \ref{dense width 0}]
The set $A$ of $\del\in S^1$, such that the geodesics $\g\in\G_+(\del)$ have a dense forward orbit in the hyperbolic unit tangent bundle of $M$, has full Lebesgue measure (the geodesic flow is ergodic with respect to the Lebesgue measure in the hyperbolic unit tangent bundle; Theorem 1.7 in \cite{walters} shows that almost all forward orbits are dense). For $\del\in A$, take any $\g\in \G_+(\del)$ and a sequence $\tau_n\in\Gamma$ positive for $\g$, such that $\tau_n\g\to \g'$, where $\g'$ is such that $\M(\g')$ contains only one minimal geodesic. Such $\g'$ exist by Corollary \ref{cor thm countable unstable}. But then $i(\g')=0$ and Lemma \ref{width semi-cont} proves that $w(\del)=0$ for all $\del\in A$.
\end{proof}

Intuitively the above proof shows that, if $\del\in S^1$ is the point at infinity of a forward dense hyperbolic geodesic, then the complicated behavior of the geodesics in $\G_+(\del)$ forces all rays in $\RR_+(\del)$ to come close to each other at infinity. Otherwise, the hyperbolic topology of $M$ would ``peel'' the geodesics away from each other, as the geodesics move around the surface.

Observe that the techniques in the proof of Theorem \ref{dense width 0} work in any dimension, as long as there is one $\g\in\G$ with $i(\g)=0$ (cf. also Remark \ref{bem high dim} in the intoduction). The existence of $\g$ with $i(\g)=0$ in general hyperbolic manifolds is, however, not clear to us.

\abs

We have two more techniques to derive $w(\g)=0$, that we will use below.

\begin{lemma}\label{one-sided approx}
Suppose $\g,\g'\in \G$ and $\tau_n\in\Gamma$ is positive for $\g$ with $\tau_n \g\to \g'$, such that $\tau_n \g\cap \g'=\emptyset$ in $X$. Then $w(\g)=0$.
\end{lemma}

\begin{proof}
First assume that $\tau_n\g$ and $\g'$ have no points at infinity in common for infinitely many $n$. The bounding geodesics $c_\g^i$ converge (after taking a subsequence if necessary) under $\tau_n$ to one bounding geodesic, $c_{\g'}^0$, say (no geodesic in $\M(\tau_n\g)$ can intersect $c_{\g'}^0$, cf. Theorem 4 in \cite{morse}). This shows that for large $n$, the two geodesics $\tau_n c_\g^i$ are close. Using that $\tau_n$ is positive for $\g$, the claim follows.

Now let $\tau_n\g(\infty)=\g'(\infty)$ for infinitely many $n$, w.l.o.g. for all $n$. Then we find $n,m$, such that $\tau:=\tau_n^{-1}\tau_m$ fixes $\g(\infty)$. Since a periodic hyperbolic geodesic cannot approach any other hyperbolic geodesic, $\g$ is not the axis of $\tau$. Theorem \ref{morse periodic} below shows that $c_\g^0,c_\g^1$ are asymptotic.

Finally, suppose $\tau_n\g(-\infty)=\g'(-\infty)$ for infinitely many $n$, w.l.o.g. for all $n$. Replace $\g$ by $\tau_1\g$, such that we can assume $\del:=\g(-\infty)=\g'(-\infty)$. Now it follows that all $\tau_n$ fix $\del$ and there exists $\tau\in\Gamma-\{\id\}$, such that $\tau_n=\tau^{k_n}$ for a sequence $k_n\in \Z$. But for such $\tau_n$, there cannot be a sequence $x_n\in \g, x_n\to \g(\infty)$, which is moved under $\tau_n$ into a given compact set, contradicting the assumption that $\tau_n$ is positive for $\g$.
\end{proof}

\begin{lemma}\label{periodic approx}
Let $\g,\g'\in\G$, $\tau_n\in\Gamma$ positive for $\g$ and $\g'$ be the axis of some $\tau\in\Gamma-\{\id\}$, such that $\tau_n\g\to \g'$ and $\tau_n\g\neq \g'$ for all $n$. Then $w(\g)=0$.
\end{lemma}

\begin{proof}
By Lemma \ref{one-sided approx}, we can assume that $\tau_n\g\cap \g'\neq\emptyset$ for all $n$. Assume that $\tau_n\g(\pm\infty)>\g'(\pm\infty)$ in the counterclockwise orientation of $S^1$, the other case being analogous. Then $c_n:=\tau_nc_\g^0$ and $c_{\g'}^1$ have to intersect in some point $x_n=c_n(s_n) = c_{\g'}^1(t_n)$. Choose $k_n\in\Z$, such that $\tau^{k_n}x_n$ lies in a compact set and assume w.l.o.g. that $x_n\to x\in c_\g^1(\R)$. The backward rays $\tau^{k_n} c_n(-\infty,s_n]$ converge w.l.o.g. to a backward ray $c:(-\infty,0]\to X$ with $c(0)=x$ and $c(-\infty)=\g'(-\infty)$ by $\tau_n\g\to \g'$. By Theorem \ref{morse periodic} (3) below, $c$ is a subray of $c_{\g'}^1$ and hence $\tau^{k_n}c_n$ converges to $c_{\g'}^1$. Note that the new sequence $\tilde\tau_n:=\tau^{k_n}\circ\tau_n$ is again positive for $\g$ (the points $x_n$ converge to $\g'(\infty)$) and under $\tilde\tau_n$, the geodesic $c_\g^1$ converges to a minimal geodesic $c^1$ lying at finite distance left of the limit $c_{\g'}^1$ of the $c_n$. But by definition of $c_\g^1$ being the leftmost minimal geodesic in $\M(\g')$, we also obtain $c^1=c_{\g'}^1$. The claim follows.
\end{proof}

\begin{remark}\label{bem periodic approx}
Assume in Lemma \ref{periodic approx}, that additionally $\tau_n\g$ and $\g'$ do not have the same point at $+\infty$ for all but finitely many $n$. Then, as all geodesics in $\G_+(\g(\infty))$ converge under $\tau_n$ to the same periodic geodesic $\g'$, the proof of Lemma \ref{periodic approx}, combined with Lemma \ref{one-sided approx}, actually shows $w(\g(\infty))=0$.

If on the other hand $\tau_n\g$ and $\g'$ do have the same point at $+\infty$, then $\g(\infty)$ is a fixed point of $\Gamma$. We will discuss this situation in the next subsection.
\end{remark}

\subsection{Periodic directions}\label{section periodic}

The following theorem due to Morse completely clarifies the structure of $\RR^\pm(\del)$, if $\del\in S^1$ is fixed by some non-trivial $\tau\in\Gamma$, cf. Figure \ref{fig periodic}. Recall the definition of the bounding geodesics $c_\g^i$ of $\M(\g)$ in Lemma \ref{bounding geodesics morse}.

\begin{thm}[Morse \cite{morse}]\label{morse periodic}
 Let $\tau\in\Gamma-\{\id\}$ be a prime element with hyperbolic axis $\g\in \G$. Then
\begin{enumerate}
 \item the set $\M_{per}(\g)$ of initial vectors to minimal geodesics in $\M(\g)$ that are invariant under $\tau$ contains the bounding geodesics $c_\g^0,c_\g^1$ and any minimal geodesic in $\M(\g)$ invariant under some $\tau^n$ with $n \in\Z-\{0\}$ is also invariant under $\tau$ itself, hence belongs to $\M_{per}(\g)$,

 \item any ray in $\RR_\pm(\g(\pm\infty))-\M_{per}(\g)$ is asymptotic to a minimal geodesic in $\M_{per}(\g)$; moreover, no geodesic in $\M(\g)$ can be asymptotic in both its senses to a single periodic minimal geodesic in $\M_{per}(\g)$,

 \item no ray $c_v$ with $v\in\RR_\pm(\g(\pm\infty))-\M_{per}(\g)$ has $\pi v\in \pi \M_{per}(\g)$,

 \item between any pair of neighboring minimal geodesics $c_0,c_1$ in the closed set $\pi \M_{per}(\g)\subset X$ ($c_0$ lying right of $c_1$ with respect to the orientation of $\g$), there are minimal geodesics $c_\pm$ in $\M(\g)$, such that $c_+(t)$ is asymptotic to $c_0$ for $t\to -\infty$ and asymptotic to $c_1$ for $t\to\infty$; $c_-$ has the opposite behavior. Writing $\M_-(\g) , \M_+(\g)$ for the sets of all initial conditions of such ``heteroclinics'', we have
\[ \M(\g) = \M_{per}(\g) \cup \M_-(\g) \cup \M_+(\g) \]
and the two sets $\pi(\M_{per}(\g) \cup \M_\pm(\g))$ are laminations of $X$.
\end{enumerate}
\end{thm}

Note that Morse proved the above statements in the Riemannian setting. However, one finds that his arguments do not rely on the Riemannian character of the metric $F$.

\begin{figure}\centering
\includegraphics[scale=0.5]{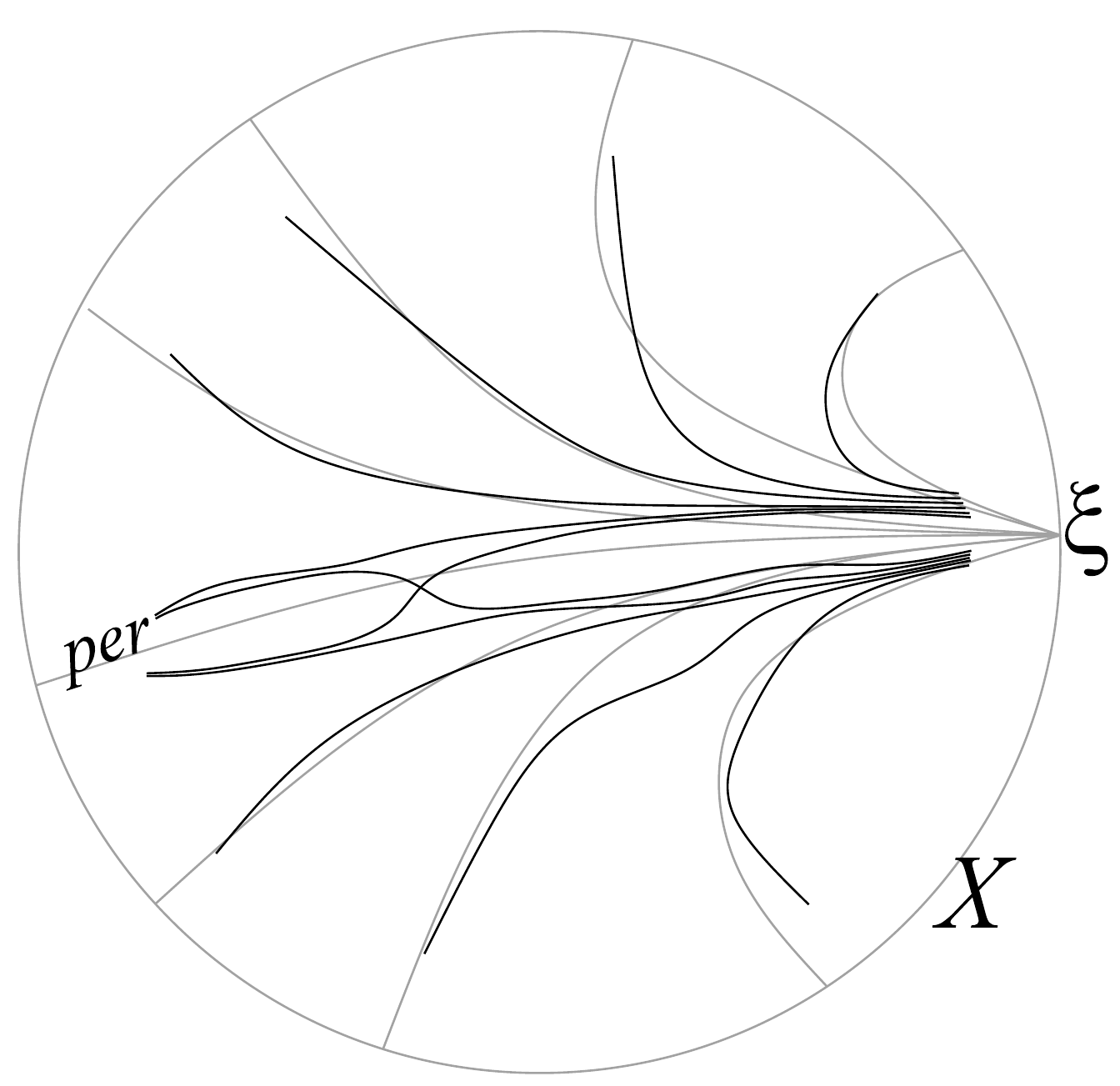}
\caption{The structure of $\RR_+(\del)$ in $X\subset \C$ in Theorem \ref{morse periodic} for $\del\in S^1$ being a fixed point of some $\tau\in\Gamma-\{\id\}$. Minimal geodesics with respect to $F$ are depicted in black, their corresponding background geodesics in gray, ``per'' labels the hyperbolic axis of $\tau$. \label{fig periodic}}
\end{figure}

Theorem \ref{morse periodic} has the following corollary, which we stated as Proposition \ref{width periodic intro} in the introduction.

\begin{cor}\label{width periodic}
Let $\del\in S^1$ be a fixed point of a group element $\tau\in\Gamma-\{\id\}$ and let $\g\in\G$ be the hyperbolic axis of $\tau$. Then
\begin{enumerate}
\item For the width we have
\begin{align*}
 w(\g) & = i(\g) = w(\del) , \\
 w(\g') & = 0 \qquad \forall \g'\in \G_+(\del)-\{\g\}.
\end{align*}
In particular, $w(\del)=0$ if and only if there is only one minimal geodesic in $\M(\g)$.

\item If $S$ is the closed strip between $c_\g^0,c_\g^1$, then
\[ \A_+(\del) = \M_{per}(\g) ~ \bigcup ~ \RR_+(\del) \cap \pi^{-1}(X-S) . \]
\end{enumerate}
\end{cor}

\begin{proof}
Theorem \ref{morse periodic} shows that any ray $c_v$ with initial condition $\pi v\in X-S$ is asymptotic to the periodic bounding geodesic $c_\g^i$ with $i=0$ or $i=1$ corresponding to the connected component of $\pi v$ in $X-S$. (1) follows and, using $w(\g')=0$, the proof of Proposition \ref{w=0 implies unique weak KAM} shows that any ray $c_v$ with $\pi v\in X-S$ is forward unstable and hence belongs to $\A_+(\del)$. Theorem \ref{morse periodic} (3) shows that also the periodic geodesics in $\M_{per}(\g)$ are forward unstable, i.e. $\A_+(\del)$ contains the set on the right hand side. Now let $v\in \RR_+(\del)-\M_{per}(\g)$, such that $\pi v\in S$. Then $\pi v$ lies in an open strip $S_0\subset S$ bounded by two neighboring periodic minimal geodesics $c_0,c_1$ in $\M_{per}(\g)$, and there are two heteroclinic minimal geodesics $c_-,c_+$ between $c_0,c_1$ (Theorem \ref{morse periodic} (4)). If $c_v(t)$ is asymptotic to $c_0$, say, and if $c_+(t)$ is asymptotic to $c_1$ as $t\to\infty$, we find $n\in\Z$, such that $\tau^nc_+(\R)$ intersects $c_v(0,\infty)$. This shows that $c_v$ is not forward unstable, hence does not belong to the Aubry set $\A_+(\del)$ by Proposition \ref{unstable Aubry}.
\end{proof}

\subsection{Recurrent directions}

In this subsection we prove Theorem \ref{thm intro recurrent} from the introduction (cf. Theorem \ref{recurrent thm 1} and its Corollaries \ref{recurrent thm 2}, \ref{recurrent thm 3}). Recall that a geodesic ray $c:[0,\infty)\to X$ is forward recurrent, if there exists a sequence $\tau_n\in\Gamma$ and a sequence of times $t_n\to \infty$, such that $\tau_n c[t_n,\infty)$ converges to $c[0,\infty)$ with respect to the euclidean Hausdorff metric in $\C$. This definition is equivalent to the usual definition of forward recurrence of the projected geodesic in $SM$.

The idea behind the results of this subsection is that the structure of $\RR_+(\del)$, if $\del$ is the endpoint of a forward recurrent geodesic $\g\in\G$, is very similar to the case where $\del$ is fixed by some $\tau\in\Gamma-\{\id\}$. If $w(\del)=0$, then $\RR_+(\del)$ has a simple structure by Proposition \ref{w=0 implies unique weak KAM}, hence we assume $w(\del)>0$.

For the whole subsection, fix $\del\in S^1$ with $w(\del)>0$.

\begin{thm}\label{recurrent thm 1}
If $\g\in \G_+(\del)$ is forward recurrent, then
\begin{enumerate}
\item there is a forward unstable geodesic in $\M(\g)$,

\item $w(\g)=i(\g)=w(\del)$.
\end{enumerate} 
\end{thm}

Presently, we do not know if $c_\g^0,c_\g^1$ are recurrent, if $\g$ is recurrent.

\begin{proof}
Let $\tau_n\in\Gamma$ be positive for $\g$, such that $\tau_n\g\to \g$. By Lemma \ref{width semi-cont}, $i(\g)\geq w(\del)$. On the other hand, $i(\g) \leq w(\g)\leq w(\del)$ by definition, so (2) follows. Moreover, the minimal geodesics $\tau_n c_\g^i$ have limits $c_i$ in $\M(\g)$ with
\[(*) \qquad  \inf_{t\in\R}d_h(c_0(\R),c_1(t)) = i(\g) \]
($\leq i(\g)$ follows from $\dot c_i\in\M(\g)$; if the infimum would be $< i(\g)$, then also $w(\g)<i(\g)$). Next, we observe that w.l.o.g., $\g$ is non-periodic and using Lemma \ref{one-sided approx}, the approximation $\tau_n\g\to \g$ can be assumed to satisfy $\tau_n\g\cap \g\neq \emptyset$ for all $n$. Assume that $\tau_n\g(\pm\infty) > \g(\pm\infty)$ in the counterclockwise orientation of $S^1$, the other case being analogous. If $c:[0,\infty)\to X$ is a ray initiating from $c_\g^0(\R)$, it cannot lie right of $c_\g^0$ by Lemma \ref{partial instability bounding}. Suppose $c$ is not a subray of $c_\g^0$, then for $\e>0$, $c[\e,\infty)$ has a uniformly positive distance from $c_\g^0(\R)$ by Lemma \ref{crossing minimals}. By $(*)$, there exist times $t_m\to \infty$, such that $d_h(c_0(t_m),c_\g^0(\R))\to 0$. For large enough $m$, we then find a point $c_0(t_m)$ and hence for large $n$ a point $x\in \tau_nc_\g^0(\R)$ in the closed strip bounded by $c_\g^0(\R), c[0,\infty)$ (using the asymptotic behavior of $\tau_nc_\g^0$). A contradiction is reached using Lemma \ref{crossing minimals}, cf. Figure \ref{fig recurrent unstable}. This shows that $c_\g^0$ is forward unstable and hence (1) holds.
\end{proof}

\begin{figure}[!htb]\centering
\includegraphics[scale=0.7]{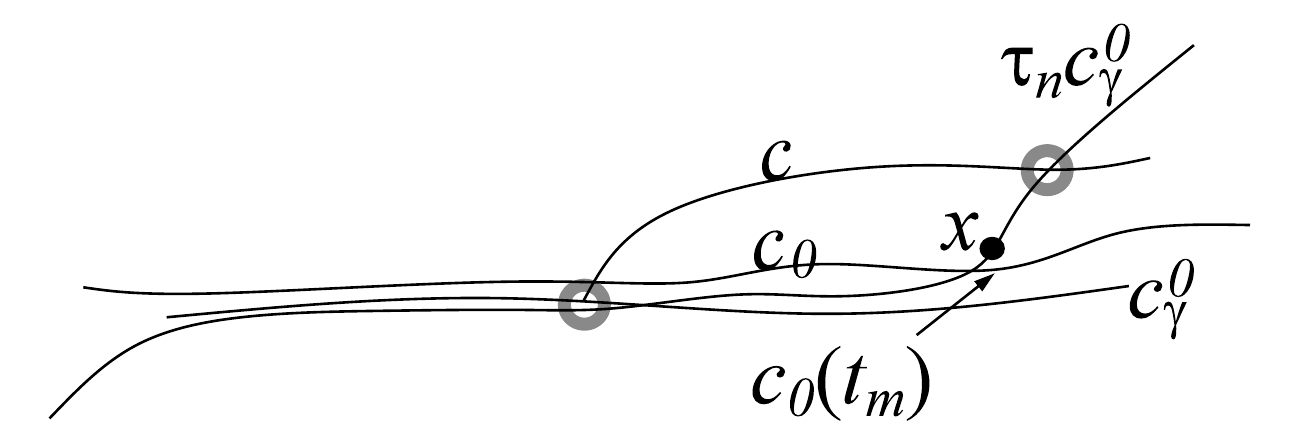}
\caption{The argument for the instability of $c_\g^0$. \label{fig recurrent unstable}}
\end{figure}

\begin{remark}\label{bem recurrent thm 1}
We saw in the proof of Theorem \ref{recurrent thm 1} that one bounding geodesic $c_\g^i$ is positively unstable, say $c_\g^0$. Let $A\subset S^1$ be the points $\del'$, such that $\g(-\infty)<\del'< \g(\infty)=\del$ in the counterclockwise order of $S^1$. Then for all $\g'\in\G_+(\del)$ with $\g(-\infty)\in A$ we have $w(\g')=0$ and hence all forward rays initiating in the connected component $X_0$ of $X-c_\g^0(\R)$ touching $A\subset S^1$ with endpoint $\del$ are forward unstable (by the arguments in the proof of Proposition \ref{w=0 implies unique weak KAM}).

Proof: Under a positive sequence $\tau_n$ for $\g$ with $\tau_n\g\to \g$, both bounding geodesics $c_{\g'}^i$ have limits in $\M(\g)$ right of the limit $c_0=\lim \tau_nc_\g^0$ by the instability of $c_\g^0$. We saw in the proof of Theorem \ref{recurrent thm 1}, that $c_0$ and $c_\g^0$ come close to each other near $+\infty$. Also the limits of $\tau_nc_{\g'}^i$ both come close to $c_\g^0$, since $\tau_nc_\g^0$ and $\tau_nc_{\g'}^i$ do not intersect. $w(\g')=0$ follows.
\end{remark}

\begin{cor}\label{recurrent thm 2}
There are at most two positively recurrent hyperbolic geode\-sics in $\G_+(\del)$.
\end{cor}

\begin{proof}
Suppose the contrary and let $\g_1,\g_2,\g_3\in \G$ with $\g_i(\infty)=\del$ all be forward recurrent, distinct and such that $\g_i(-\infty)$ are increasing with respect to the counterclockwise orientation of $S^1$. But Remark \ref{bem recurrent thm 1} shows that $w(\g')=0$ for all $\g'$ on one side of $\g_2$, contradicting $w(\g_1)=w(\g_3)=w(\del)>0$ in Theorem \ref{recurrent thm 1} (2).
\end{proof}

The next corollary shows that for all but countably many $\del\in S^1$ with $w(\del)>0$, such that there is a forward recurrent $\g\in\G_+(\del)$, the set $\RR_+(\del)$ has the same structure as in the case where $\del$ is fixed by some $\tau\in\Gamma-\{\id\}$, cf. Corollary \ref{width periodic} (the only difference being that there can be more forward unstable rays between $c_\g^0,c_\g^1$).

\begin{cor}\label{recurrent thm 3}
If $\del$ satisfies the conclusion of Theorem \ref{thm countable unstable} (i.e. all bounding geodesics $c_\g^i$ with $\g(\infty)=\del$ are forward unstable), then there can be at most one forward recurrent direction in $\G_+(\del)$ and if there is a forward recurrent direction $\g\in \G_+(\del)$, then
\begin{enumerate}
\item for the width we have
\begin{align*}
 w(\g) & = i(\g) = w(\del), \\
 w(\g') & = 0 \qquad \forall \g'\in \G_+(\del)-\{\g\},
\end{align*}

\item if $S$ is the open strip between $c_\g^0,c_\g^1$, then
\[ \A_+(\del) ~\supset ~ \RR_+(\del) \cap \pi^{-1}( X-S) . \]
\end{enumerate}
\end{cor}

\begin{proof}
We now by Corollary \ref{recurrent thm 2} that there can be at most two recurrent directions $\G_+(\del)$, which both have maximal width. Arguing as in Remark \ref{bem recurrent thm 1} and using our assumption that the bounding geodesics of the two directions cannot intersect, the width of one of the two recurrent directions would vanish. Items (i) and (ii) follow from Theorem \ref{recurrent thm 1} and Remark \ref{bem recurrent thm 1}. 
\end{proof}

\subsection{Simple directions}

In this short subsection we prove Theorem \ref{thm simple}, i.e. if $\del\in S^1$ is not a fixed point of $\Gamma$ and $\G_+(\del)$ contains a geodesic $\g$, which is simple (i.e. $\g,\tau \g$ are disjoint for all $\tau\in\Gamma-\{\id\}$), then $w(\del)=0$.

\begin{proof}[Proof of Theorem \ref{thm simple}]
Let $\g\in\G$ be simple and $\del=\g(\infty)$, such that $\del$ is not fixed under any $\tau\in\Gamma-\{\id\}$. Choose any positive sequence $\tau_n\in\Gamma$, such that $\tau_n\g$ converges to some $\g'\in\G$. Since $\g$ is simple, also $\g'$ is simple. If $\g'$ is the axis of some $\tau\in\Gamma-\{\id\}$, apply Lemma \ref{periodic approx} and Remark \ref{bem periodic approx}, which shows $w(\del)=0$ in this case. If $\g'$ is not an axis, there exist $\tau_n'\in\Gamma$ positive for $\g'$, such that $\tau_n'\g'$ converge to some $\g''$ with $\tau_n'\g'\cap \g''=\emptyset$, since $\g'$ is simple. Apply Lemma \ref{one-sided approx}, showing $w(\g')=0$. Now apply Lemma \ref{width semi-cont} to obtain $w(\del) \leq i(\g')\leq w(\g')=0$.
\end{proof}

\subsection{Proof of Theorem \ref{thm A}}\label{section (A)}

On the 2-torus, the only directions that admit more than one weak KAM solution are the rational directions, i.e. those which have a periodic geodesic.

\abs

{\bf Question.} Do we have $w(\del)=0$ for all $\del\in S^1$ not being fixed under any non-trivial group element of $\Gamma$?\abs

In this subsection we show that the answer is affirmative, if we make an additional assumption.

\abs

{\bf Assumption (A).} For any $\del\in S^1$, the union $L(\del)$ of all bounding geodesics $c_\g^0,c_\g^1$ with $g\in\G_+(\del)$ is a lamination of $X$, i.e. the bounding geodesics pairwise do not intersect. \abs

(A) is fulfilled, e.g., if $F$ has no conjugate points (in this case all geodesics are unstable, cf. Theorem 12.1 in \cite{morse_hedlund}) and Theorem \ref{thm countable unstable} shows that (A) is always ``almost'' fulfilled.

\begin{thm}\label{thm assumption A}
If $F$ satisfies the assumption (A), then $w(\del)=0$ for all $\del\in S^1$ not being fixed under any non-trivial group element of $\Gamma$.
\end{thm}

The idea in the following proof is similar to an idea in Section 3 of \cite{coudene}.

\begin{proof}
\begin{figure}[!htb]\centering
\includegraphics[scale=0.63]{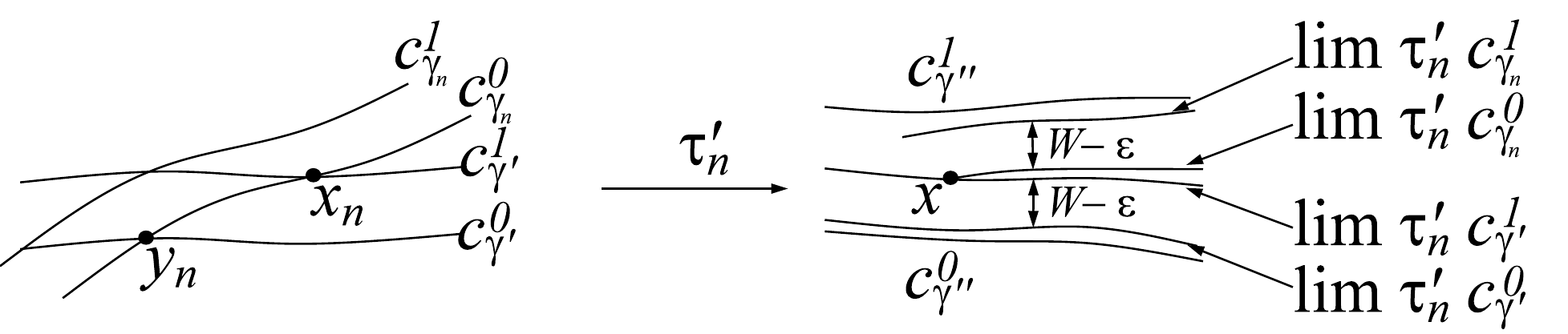}
\caption{The objects in the proof of Theorem \ref{thm assumption A}. \label{fig aussmption_A}}
\end{figure}

Figure \ref{fig aussmption_A} depicts the following arguments. Set
\begin{align*}
W := \sup\{ w(\g) : \text{$\g\in \G$, such that $\g$ is not the axis of any $\tau\in\Gamma-\{\id\}$} \} ,
\end{align*}
assume $W>0$ and choose a small $\e>0$ and $\g\in\G$ with
\[ W-\e\leq w(\g)\leq W. \]
Let $\tau_n\in\Gamma$ be a positive sequence for $\g$, such that $\g_n := \tau_n\g\to \g'$ for some $\g'\in\G$. We may assume w.l.o.g., that $\g_n\cap \g'\neq \emptyset$ and $\g'$ is not an axis, since otherwise $w(\g)=0$ (Lemmata \ref{one-sided approx} and \ref{periodic approx}). We assume moreover that $\g_n(\pm\infty)>\g'(\pm\infty)$ for all $n$ in the counterclockwise orientation of $S^1$, the other case being analogous. By Lemma \ref{width semi-cont}, we have
\[ W \geq w(\g') \geq i(\g') \geq w(\g) \geq W-\e. \]
There is a point of intersection $x_n$ of the geodesics $c_{\g'}^1(\R), c_{\g_n}^0(\R)$. Let $\tau_n'\in\Gamma$ be a positive sequence for $\g'$, such that $\tau_n'\g'\to \g''\in \G$ and such that $\tau_n'x_n$ converges to some $x\in X$. Let $y_n$ be a point of intersection of $c_{\g'}^0(\R), c_{\g_n}^0(\R)$, then $d_h(x_n,y_n)\to\infty$ (the limit of $c_{\g_n}^0$ lies left of $c_{\g'}^0$). For $n\to\infty$, the segment of $\tau_n'c_{\g_n}^0$ between $\tau_n'x_n,\tau_n'y_n$ becomes a backward ray in the strip between $\lim \tau_n'c_{\g'}^i$ in $\M(\g'')$ and since by the assumption (A), no minimal geodesics of different type can intersect, the limit of $\tau_n'c_{\g_n}^0$ also belongs to $\M(\g'')$. But now we have in the forward end of $\M(\g'')$ to strips of width $\geq W-\e$, one bounded by $\lim\tau_n'c_{\g'}^i$ and the other bounded by $\lim\tau_n'c_{\g_n}^i$, showing $w(\g'')\geq 2(W-\e)$ and hence $\g''$ is an axis. Lemma \ref{periodic approx} then shows $w(\g')=0$, a contradiction.
\end{proof}

\appendix

\section{Weak KAM solutions in dimension two}\label{appendix}

In this appendix we prove results about weak KAM solutions, which are special in dimension two. Even though the results are not used in the paper, the techniques might prove useful in the future. As before, we only assume that $F,g_h$ are uniformly equivalent for the Morse Lemma to hold.

Our first lemma is true in any dimension, replacing $S^1$ with the so-called Gromov boundary.

\begin{lemma}\label{J(u) semi-cont}
The sets $\H_+\cap \{u(x_0)=0\}$ for fixed $x_0\in X$ are sequentially compact in the $C_{loc}^0$ topology and if $u,u_n\in \H_+$ with $u_n\to u$ in $C_{loc}^0$, then
\[ \lim_{n\to\infty} \J_+(u_n) := \{ v\in SX ~|~ \exists v_n\in\J_+(u_n) : v_n\to v \} = \J_+(u). \]
Moreover, the asymptotic direction $\om:\H_+\to S^1$ is continuous.
\end{lemma}

\begin{proof}
By $u(y)-u(x)\leq d_F(x,y)\leq c_F\cdot d_h(x,y)$ with $c_F$ as in the proof of Lemma \ref{existence horofunction}, all $u\in\H_+$ are equi-Lipschitz. It now follows from the Arzela-Ascoli Theorem that any sequence $u_n\in \H_+$ has a convergent subsequence with limit in $C^0(X)$. Assume $u_n\to u$. If $x,y\in X$, then
\[ u(y)-u(x) = \lim u_n(y)-u_n(x) \leq \lim d_F(x,y) = d_F(x,y) . \]
Let now $v_n\in\J_+(u_n)$, such that $v_n\to v$ for some $v\in SX$. Then by $|u_n(x)-u_n(y)|\leq c_F d_h(x,y)$ and $c_{v_n}(t)\to c_v(t)$, we have $u_n(c_v(t))-u_n(c_{v_n}(t))\to 0$ and hence
\begin{align*}
& u(c_v(t))-u(\pi v) = \lim u_n(c_v(t))-u_n(\pi v) \\
= ~&\lim u_n(c_v(t))-u_n(c_{v_n}(t)) + u_n(c_{v_n}(t)) -u_n(\pi v) + u_n(\pi v_n)- u_n(\pi v_n) \\
= ~&\lim u_n(c_{v_n}(t)) - u_n(\pi v_n) = t.
\end{align*}
Fixing $x=\pi v_n$ shows that for all $x$ we find a ray $c_v:[0,\infty)\to X$ with $\pi v=x$, such that $u(c_v(t))-u(x)=t$. Moreover, if we take a subsequence of $u_n$, such that $\om(u_n)\to \del$ for some $\del\in S^1$, then Lemma \ref{as dir cont rays} shows that $c_v(\infty)=\del$ for any so obtained $v$ and hence $u\in\H_+(\del)$.

The same arguments show that $\lim \J_+(u_n)\subset \J_+(u)$. Conversely, if $v\in \J_+(u)$, let $t>0$ and $v_n\in\J_+(u_n)$, such that $\pi v_n = c_v(t)$. Lemma \ref{lemma fathi} shows $\lim v_n=\dot c_v(t)$ and since $\lim \J_+(u_n)$ is closed, we have $v\in \lim \J_+(u)$, hence $\lim \J_+(u_n) = \J_+(u)$.

To show the continuity of $\om$, let $u_n\to u$, then $\om(u_n)=c_{v_n}(\infty)$ for any $v_n\in\J_+(u_n)$ and any limit $v$ of $\{v_n\}$ lies in $\J_+(u)\subset \RR_+(\om(u))$. Lemma \ref{as dir cont rays} shows that $\om(u_n)=c_{v_n}(\infty) \to c_v(\infty)=\om(u)$.
\end{proof}

Due to $\dim M=2$, we can always find special weak KAM solutions. They are linked to the bounding geodesics $c_\g^0,c_\g^1$ of $\M(\g)$ from Lemma \ref{bounding geodesics morse}.

\begin{prop}[Bounding weak KAM solutions] \label{bounding horofunctions}
Let $x_0\in X$. For each $\del\in S^1$, there exist two unique $u_+^0(\del),u_+^1(\del)\in\H_+(\del)\cap \{u(x_0)=0\}$ with the following property: for all sequences $u_n\in \H_+\cap \{u(x_0)=0\}$ with $\om(u_n)\to \del$ and $\om(u_n)\neq \del$, any $C_{loc}^0$ limit lies in $\{u_+^0(\del),u_+^1(\del)\}$. More precisely, assuming the counterclockwise orientation of $S^1$, we have
\[ \lim_{n\to\infty}u_n = \begin{cases} u_+^0(\del) & : \om(u_n) < \del \\ u_+^1(\del) & : \om(u_n) > \del \end{cases}. \]
Analogously, for $u_n\in \H_-$ with $\om(u_n)\to\del$ we have
\[ \lim_{n\to\infty}u_n = \begin{cases} u_-^1(\del) & : \om(u_n) < \del \\ u_-^0(\del) & : \om(u_n) > \del \end{cases}. \]
\end{prop}

One can easily deduce, that for the forward Aubry set of direction $\del$ from Subsection \ref{section unstable} we have
\[ \A_+(\del) = \J_+(u^0_+(\del))\cap \J_+(u^1_+(\del)). \]

\begin{proof}
\begin{figure}[htb!]\centering
\includegraphics[scale=0.7]{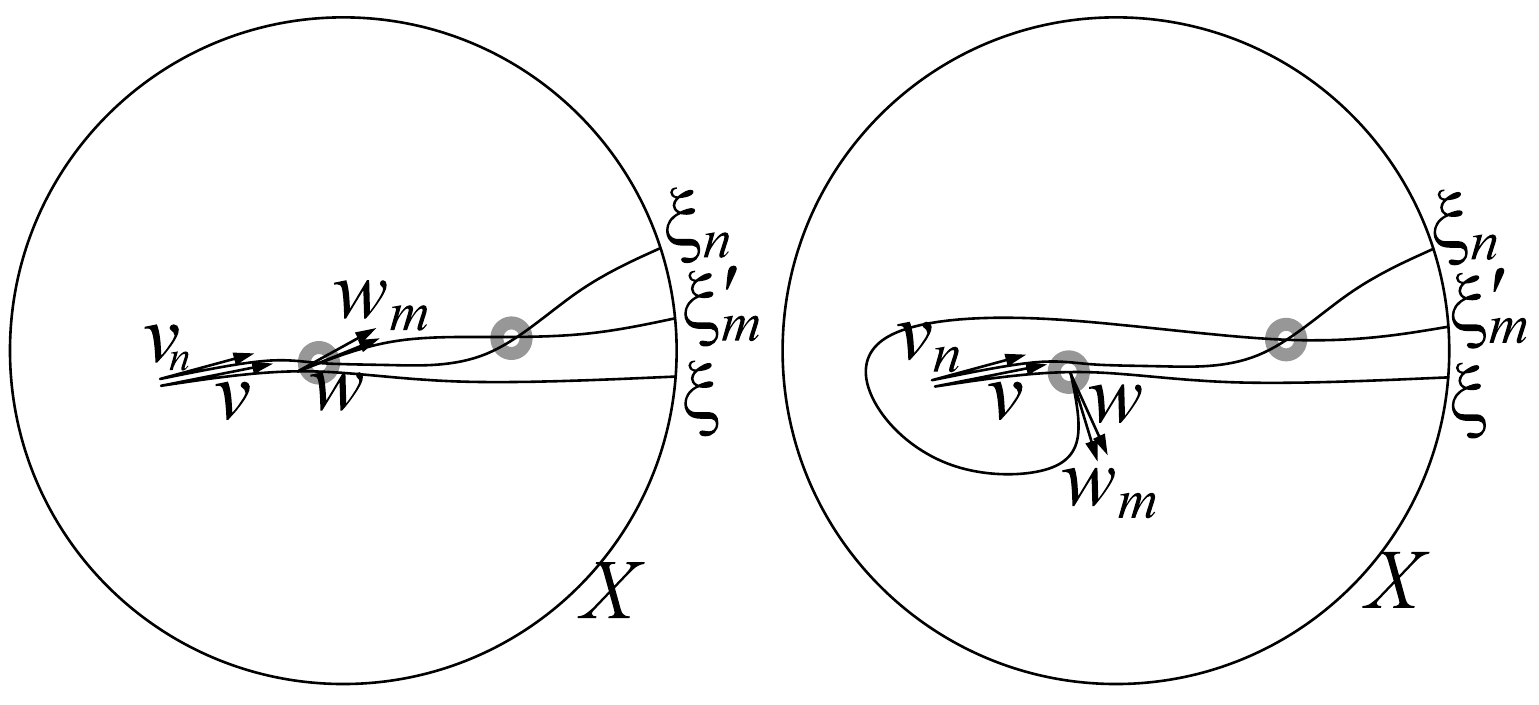}
\caption{The argument in the proof of Proposition \ref{bounding horofunctions}. Left: $w$ starts off $c_v$ to the left. Right: $w$ starts off $c_v$ to the right. The gray circles indicate the (almost) intersections that contradict Lemma \ref{crossing minimals}. \label{fig bounding_KAM}}
\end{figure}

Let $\del_n,\del_n'\to\del$ and assume that in the counterclockwise orientation of $S^1$, we have $\del_n,\del_n'> \del$. Choose any forward weak KAM solutions $u_n\in \om^{-1}(\del_n), u_n'\in \om^{-1}(\del_n')$ vanishing in a fixed reference point $x_0\in X$ and let $u, u'\in \H_+(\del)$ be limit functions of the $u_n, u_n'$, respectively (after passing to a subsequence, using continuity of $\om$ and compactness of $\H_+\cap\{u(x_0)=0\}$). We claim that $\J_+(u)=\J_+(u')$, which also shows $u=u'$ by Corollary \ref{J=J' then u=u'}. By symmetry, we need to show only $\J_+(u)\subset \J_+(u')$.

Let $v\in\J_+(u)$ and $t>0$. We claim $\dot c_v(t)\in\J_+(u')$, which shows $v\in \J_+(u')$, since $t>0$ is arbitrary and $\J_+(u')$ is closed. By Lemma \ref{J(u) semi-cont}, there exist $v_n\in\J_+(u_n)\subset \RR_+(\del_n)$ with $v_n\to v$. By $\pi\J_+(u_n')=X$, we find analogously $w_n\in\J_+(u_n')\subset\RR_+(\del_n')$ with $\pi w_n=c_v(t)$, converging to a vector $w\in\J_+(u')$. Using Lemma \ref{crossing minimals}, the sketches in Figure \ref{fig bounding_KAM} show a contradiction by considering two cases, if $w\neq \dot c_v(t)$.
\end{proof}

\begin{remark}
The above proof shows: If $v,v_n\in\RR_+$, such that $v_n\to v$ with $c_{v_n}(\infty)\neq c_v(\infty)=:\del$, then $v\in \J_+(u_+^0(\del))\cup \J_+(u_+^1(\del))$. One can think of $\J_+(u_+^0(\del))$ as the set of rightmost vectors in $\RR_+(\del)$, $\J_+(u_+^1(\del))$ as the leftmost vectors in $\RR_+(\del)$.
\end{remark}

We show how $u_+^i(\del)$ are related to the bounding geodesics in Lemma \ref{bounding geodesics morse}.

\begin{prop}\label{bounding calibrated}
If $\del\in S^1$, $\g\in\G_+(\del)$ and $i\in\{0,1\}$, then the bounding geodesic $c_\g^i$ is $u_-^i(\del)$- and $u_+^i(\del)$-calibrated.
\end{prop}

\begin{proof}
E.g. for $x=c_\g^1(t)$, consider any sequence of $v_n\in \RR_+(\del_n)\cap T_xX$, where $\del_n\to \del$, such that $\del_n>\del$ in the counterclockwise orientation of $S^1$. Proposition \ref{bounding horofunctions} and Lemma \ref{J(u) semi-cont} show that $v:= \lim v_n\in \J(u_+^1(\del))$ and if $v\neq \dot c_\g^1(t)$, $c_v$ has to lie asymptotically left of $c_\g^1$, which is prohibited by Lemma \ref{partial instability bounding}.
\end{proof}

We close by the following nice property of the bounding weak KAM solutions. Assume that $F$ is invariant under $\Gamma$ and let $p:X\to M$ be the covering map with differential $p_*:TX\to TM$. Write $\Om$ for the non-wandering set of $\phi_F^t$ in $p_*\M\subset SM$.

\begin{cor}
 The non-wandering set $\Om$ is contained in
\[ \bigcup_{\del\in S^1} \bigcup_{i=0,1} p_*\J(u_+^i(\del)) . \]
\end{cor}

Hence, if one is interested only in recurrent dynamics (e.g. in the supports of invariant measures), one can restrict the attention to the dynamics in the above set, which in the universal cover consists of only two $\phi_F^t$-invariant graphs over $0_X$ for each direction $\del\in S^1$. In another paper \cite{minimal_preprint}, we exploited this idea to calculate the topological entropy of $\phi_F^t|_\M$.

\begin{proof}
 If $v\in\Om$, then there exist by definition $w_n\in p_*\M$ and $t_n\to\infty$, such that $w_n,\phi_F^{t_n}w\to v$ in $SM$. Lifted to $X$, we find $\tau_n\in\Gamma$, such that $\tilde w_n$ and $(\tau_n)_*\phi_F^{t_n}\tilde w_n$ converge to $\tilde v$. If the asymptotic direction $\del:=c_{\tilde v}(\infty)$ is a fixed point of $\Gamma$, then $\tilde v$ belongs to the set above by the results in Subsecction \ref{section periodic}. In the other case, if $\del$ is not fixed by $\Gamma$, the asymptotic directions of $c_{\tilde w_n}$ and $c_{(\tau_n)_*\phi_F^{t_n}\tilde w_n}$ cannot all be equal to $\del$ by the presence of the non-trivial $\tau_n$. Proposition \ref{bounding horofunctions} now proves the claim.
\end{proof}


\end{document}